\documentclass[12pt]{article}
\usepackage{a4}
\usepackage{amsthm}
\usepackage{amsmath}
\usepackage{amssymb}
\usepackage{amsfonts}
\usepackage{stmaryrd}
\usepackage{cite}
\usepackage{epsfig}
\newtheorem{conjecture}{Conjecture}
\newtheorem{theorem}{Theorem}
\newtheorem{lemma}[theorem]{Lemma}
\newtheorem{proposition}[theorem]{Proposition}

\newtheorem{corollary}[theorem]{Corollary}

\newcommand\jj{\vec j}
\newcommand\ff{\tilde f}
\newcommand\JJ{{\mathbb J}}
\newcommand\NN{{\mathbb N}}
\newcommand\RR{{\mathbb R}}
\newcommand{\dd}{\;\mathrm{d}}
\DeclareMathOperator{\Real}{Re}
\DeclareMathOperator{\Tr}{Tr}
\newcommand\unit{\iota}
\newcommand\spectrum{{Spectrum}}

\begin{document}
\title{Cycles of length three and four in tournaments\thanks{This work was supported by the European Research Council (ERC) under the European Union's Horizon 2020 research and innovation programme (grant agreement No 648509) and by the Engineering and Physical Sciences Research Council Standard Grant number EP/M025365/1. This publication reflects only its authors' view; the European Research Council Executive Agency is not responsible for any use that may be made of the information it contains.}}

\author{Timothy F.~N. Chan\thanks{School of Mathematical Sciences, Monash University, Melbourne 3800, Australia, and Mathematics Institute and DIMAP, University of Warwick, Coventry CV4 7AL, UK. Email: {\tt timothy.chan@monash.edu}.}\and
        Andrzej Grzesik\thanks{Faculty of Mathematics and Computer Science, Jagiellonian University, {\L}ojasiewicza 6, 30-348 Krak\'{o}w, Poland. E-mail: {\tt Andrzej.Grzesik@uj.edu.pl}.}\and
        Daniel Kr{\'a}l'\thanks{Faculty of Informatics, Masaryk University, Botanick\'a 68A, 602 00 Brno, Czech Republic, and Mathematics Institute, DIMAP and Department of Computer Science, University of Warwick, Coventry CV4 7AL, UK. E-mail: {\tt dkral@fi.muni.cz}.}\and
	Jonathan A. Noel\thanks{Mathematics Institute and DIMAP, University of Warwick, CV4 7AL Coventry, UK. Research of the fourth author is supported by the Leverhulme Trust Early Career Fellowship ECF-2018-534. E-mail: {\tt J.Noel@warwick.ac.uk}.}}
	
\date{}
\maketitle
\begin{abstract}
Linial and Morgenstern conjectured that, among all $n$-vertex tournaments with $d\binom{n}{3}$ cycles of length three,
the number of cycles of length four
is asymptotically minimized by a random blow-up of a transitive tournament
with all but one part of equal size and one smaller part.
We prove the conjecture for $d\ge 1/36$ by analyzing the possible spectrum of adjacency matrices of tournaments.
We also demonstrate that the family of extremal examples is broader than expected and
give its full description for $d\ge 1/16$.
\end{abstract}

\section{Introduction}
\label{sec-intro}

One of the oldest theorems in extremal graph theory is Mantel's theorem~\cite{Man07},
which asserts that every $n$-vertex graph with more than $n^2/4$ edges contains a triangle.
The Erd\H os--Rademacher Problem, which can be traced back to the work of Rademacher in the 1940's and
the later work of Erd\H os~\cite{Erd55}, asks for the minimum possible number of triangles in a 
graph with a given number of vertices and edges. It was conjectured that this minimum is asymptotically 
attained by a complete multipartite graph (i.e. a blow-up of a constant-sized clique)
with all but one part of equal size and one smaller part. 
This conjecture attracted substantial attention for several decades,
see e.g.~\cite{Bol76,Fis89,Goo59,LovS83},
until its solution by Razborov~\cite{Raz08} in 2008 using his newly developed flag algebra method.
Pikhurko and Razborov~\cite{PikR17} described the asymptotic structure of all
extremal graphs for this problem and an exact description was obtained by Liu, Pikhurko and Staden~\cite{LiuPS17}.
The more general problem of determining the minimum asymptotic density of $k$-cliques in graphs with given edge-density 
(the Erd\H os--Rademacher Problem corresponds to the case $k=3$) has also been solved
by Nikiforov~\cite{Nik11} (the case $k=4$) and by Reiher~\cite{Rei16} in full generality.

In this paper, we investigate a related problem for tournaments posed by Linial and Morgenstern~\cite{LinM16}, who asked for the minimum density of 4-cycles in a large tournament with fixed density of 3-cycles.
They conjectured that the tournament asymptotically minimizing this density is a blow-up of a transitive tournament with
all but one part of equal size and one smaller part in which the arcs within each part are oriented randomly 
(they call this construction a \emph{random blow-up}), i.e., the structure of the conjectured 
extremal examples is akin to those of the Erd\H os--Rademacher problem.

We confirm this conjecture in the case where the proposed extremal examples have two or three parts and provide a full description of extremal tournaments in the two-part case.
In contrast to many of the recent proofs in this area that use the
flag algebra method, our approach is based on the 
analysis of the spectrum of adjacency matrices of tournaments.

We now state the problem that we study in the paper more formally.
The \emph{density} of the directed cycle $C_{\ell}$ of length $\ell$ in a tournament $T$, denoted by $t(C_{\ell},T)$,
is the probability that a random mapping from $V(C_{\ell})$ to $V(T)$ is a homomorphism 
(i.e. arcs of $C_\ell$ map to arcs of $T$).
Note that, for fixed $\ell$, a tournament $T$ on $n$ vertices contains $t(C_\ell,T)n^\ell/\ell+O(n^{\ell-1})$ 
cycles of length $\ell$. In fact, for $\ell \in \{3,4,5\}$, the error term is zero as every homomorphism of 
$C_\ell$ to $T$ is injective. A standard application of the Cauchy--Schwarz inequality shows that 
$t(C_3,T)\le 1/8$ for every tournament $T$ (see~\cite[Fact 1]{ChuG91} for details). 
Our focus is on bounding the minimum possible value of $t(C_4,T)$ asymptotically as a function of $t(C_3,T)$. 

\begin{figure}
\begin{center}
\epsfbox{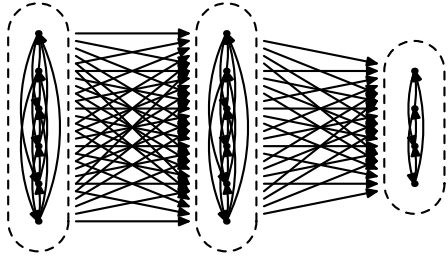}
\end{center}
\caption{An illustration of the random blow up construction for $z=3/8$.}
\label{fig-z}
\end{figure}

Next, we describe the family of conjectured tight examples from~\cite{LinM16} which will motivate the definition of the function $g$ below. 
Given $z\in [0,1]$, we let $n$ be an integer chosen large with respect to $z$ and define an $n$-vertex tournament $T$ as follows. 
If $z=0$, then let $T$ be a transitive tournament. Otherwise, the vertices of $T$ are split into $\lfloor z^{-1}\rfloor+1$ disjoint parts 
$V_1,\ldots,V_{\lfloor z^{-1}\rfloor+1}$ such
that $\lfloor z^{-1}\rfloor$ parts contain exactly $\lfloor zn\rfloor$ vertices and
the remaining part contains the rest of the vertices (note that if $z^{-1}$ and $zn$ are integers, then the last part is empty).
If two vertices $v$ and $v'$ respectively belong to distinct parts $V_i$ and $V_j$ with $i<j$,
then the tournament $T$ contains an edge from $v$ to $v'$.
If $v$ and $v'$ instead belong to the \textit{same} part, then the edge between them
is oriented from $v$ to $v'$ with probability $1/2$,
i.e., the vertices of each part induce a randomly oriented tournament. See Figure~\ref{fig-z} for 
an illustration. It is easy to see that $t(C_3,T)=t(C_4,T)=0$ if $z=0$ and, if $z\in (0,1]$, then, with 
high probability, it holds that
\[t(C_3,T) = \frac{1}{8}\left(\lfloor z^{-1}\rfloor z^3+\left(1-\lfloor z^{-1}\rfloor z\right)^3\right) + o(1)\]
and
\[t(C_4,T) = \frac{1}{16}\left(\lfloor z^{-1}\rfloor z^4+\left(1-\lfloor z^{-1}\rfloor z\right)^4\right) + o(1)\]
because of the concentration around the expected values.

The conjecture of Linial and Morgenstern~\cite{LinM16} asserts that the above construction is asymptotically optimal. In light of this, we write \textit{the regime of $k$ parts} to denote the set of values of $t(C_3,T)$ between $1/(8k^2)$ and $1/(8(k-1)^2)$, corresponding to the range of values for which the above construction has its vertices split into $k$ parts. In particular, the focus of this paper is on the regimes of two and three parts, which refer to values of $t(C_3,T)$ in the ranges $[1/32,1/8]$ and $[1/72,1/32]$, respectively.

To formally state the conjecture, define $g:[0,1/8]\to[0,1]$ as follows: $g(0)=0$ and 
\[g\left(\frac{1}{8}\left(\lfloor z^{-1}\rfloor z^3+\left(1-\lfloor z^{-1}\rfloor z\right)^3\right)\right)=
  \frac{1}{16}\left(\lfloor z^{-1}\rfloor z^4+\left(1-\lfloor z^{-1}\rfloor z\right)^4\right)\]
for $z\in (0,1]$.
\begin{conjecture}[{Linial and Morgenstern~\cite[Conjecture 2.2]{LinM16}}]
\label{conj-main}
It holds that \[t\left(C_4,T\right)\ge g\left(t(C_3,T)\right)+o(1)\] for every tournament $T$.\footnote{Linial and Morgenstern phrased 
	their conjecture equivalently in terms of subgraph counts as opposed to homomorphism densities; 
	the quantities $c_3$ and $c_4$ in their paper are asymptotically equal to $2t(C_3,T)$ 
	and $6t(C_4,T)$, respectively.}
\end{conjecture}
Although $0\leq t(C_3,T)\leq 1/8$ for every tournament $T$, the conjecture is currently only known to hold for tournaments with 3-cycle density asymptotically equal to $0$, $1/8$, or $1/32$~\cite{LinM16}.

\begin{figure}[t]
\begin{center}
\epsfbox{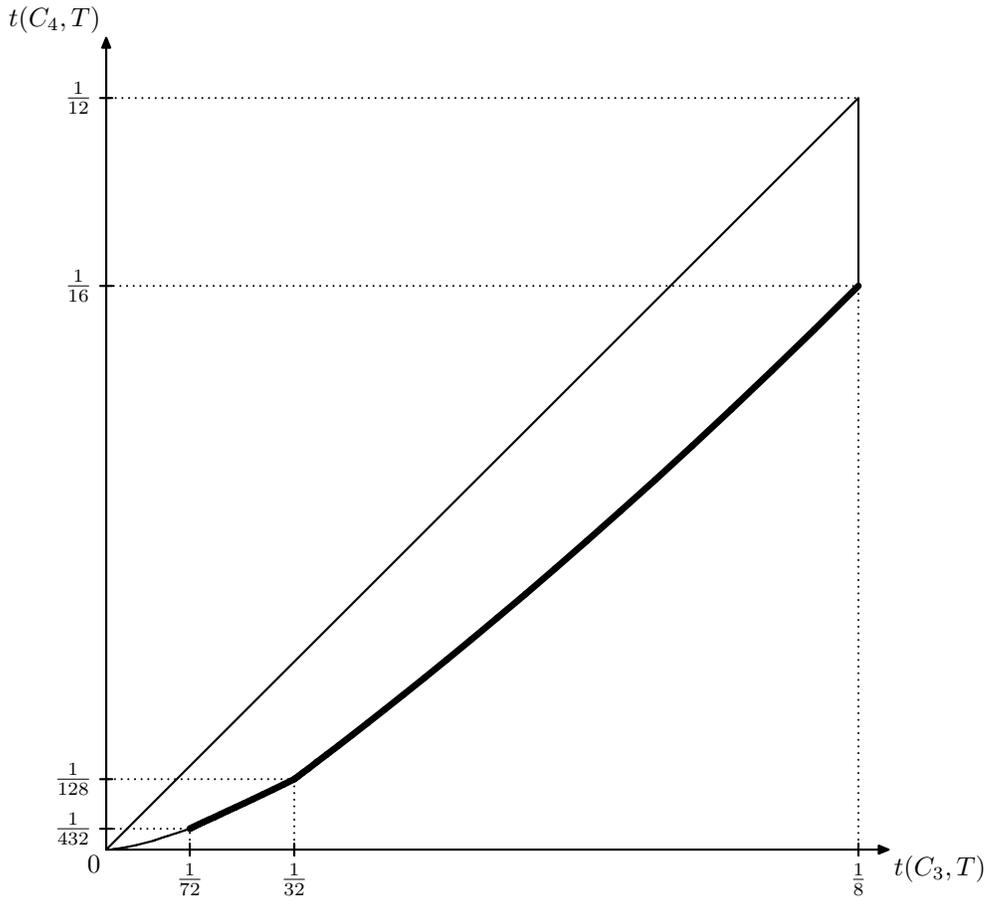}
\end{center}
\caption{The conjectured region of asymptotically feasible densities of $C_3$ and $C_4$ in tournaments.
        The lower bound for $t(C_3,T)\in\{1/8,1/32\}$ and the upper bound were proved in~\cite{LinM16}.
         The rest of the lower bound is conjectured except for the part depicted in bold, which we prove in this paper.}
\label{fig-dens34}
\end{figure}

We confirm the conjecture for all $3$-cycle densities in the range $[1/72,1/8]$ (Theorems~\ref{thm-reg2} 
and~\ref{thm-reg23}, see the discussion after Conjecture~\ref{conj-matrix} in Section~\ref{sec-prelim}) and
characterize the asymptotic structure of the extremal tournaments for densities in $[1/32,1/8]$ 
(Corollary~\ref{cor-reg2}).
We refer the reader to Figure~\ref{fig-dens34} for the visualization of the 
conjectured feasible region of $3$-cycle and $4$-cycle densities. 

Conjecture~\ref{conj-main} appears to be resistant to the flag algebra method and
we follow a different approach based on spectral analysis of adjacency matrices of tournaments.
We believe that the difficulty in applying the flag algebra method is rooted in the fact that random 
blow-ups of transitive tournaments are far from being the unique extremal examples for Conjecture~\ref{conj-main}.
In particular, a rather complicated family of extremal examples $T$ is described as follows.
Denote the vertices of $T$ by $v_1,\dots,v_n$ and associate $v_i$ with a real number $p_i\in [0,1/2]$, 
$i=1,\ldots,n$. Then, direct the edge $v_iv_j$ from $v_i$ to $v_j$ with probability $1/2+p_i-p_j$.
Note that, if all the values of $p_i$ are either $0$ or $1/2$, then
this construction is nothing more than a random blow-up of a $2$-vertex tournament, i.e., it is identical to the 
examples of~\cite{LinM16} for $3$-cycle density in $[1/32,1/8]$.
For large $n$, this tournament satisfies $t(C_4,T)=g(t(C_3,T))+o(1)$ with high probability (this follows from
Theorem~\ref{thm-reg2}). In particular, all tournaments obtained in this way are extremal with respect to 
Conjecture~\ref{conj-main} in the regime of two parts.
In Corollary~\ref{cor-reg2}, we prove that these are asymptotically the only extremal constructions in this regime.
As we have said, we believe that this complex structure of extremal examples makes it challenging to apply the flag algebra method.
Indeed, the method relies on finding suitable positive semidefinite matrices such that
the vectors of rooted homomorphism densities belong to their kernels for all choices of roots in any extremal example.
However, the richness of the structure of extremal examples restricts significantly
which matrices could possibly appear in the flag algebra argument.

We conclude this introductory section by summarizing the previous results
on the interplay between the densities of $C_3$ and $C_4$ in tournaments.
Linial and Morgenstern~\cite{LinM16} proved $t(C_4,T)\geq\frac{12t(C_3,T)^2}{1+16t(C_3,T)}$ 
which confirmed Conjecture~\ref{conj-main} in the case $t(C_3,T)=1/32+o(1)$. 
They also proved that, for $d\in[0,1/8]$, the asymptotically feasible densities of cycles of length 
four in tournaments with $t(C_3,T)=d+o(1)$ form an interval~\cite[Proof of Lemma 1.3]{LinM16}. 
For the related problem of \textit{maximizing} the $4$-cycle density relative to $t(C_3,T)$, they proved that
$t(C_4,T)\leq 2t(C_3,T)/3$ for all $T$ where equality holds if and only if $|V(T)|\geq4$ 
and every set of $4$ vertices either induces a transitive tournament or contains a 
$4$-cycle.

For completeness, we briefly describe a tight family of tournaments from~\cite{LinM16} 
for the upper bound $t(C_4,T)\leq 2t(C_3,T)/3$. Let $\xi\in [0,1/2]$ and denote the 
vertices by $v_1,\ldots,v_n$ with $n$ large. The edge between $v_i$ and $v_j$, $i\le j$, 
is directed from $v_i$ to $v_j$ if and only if $j\leq i+ (1-\xi)n$. If $\xi=0$, 
we obtain a transitive tournament and if $\xi=1/2$, we obtain the ``circular'' 
$n$-vertex tournament, i.e., the tournament contains edges from $v_i$ to 
$v_{i+1},\ldots,v_{i+\lfloor n/2\rfloor}$ (indices modulo $n$). These tournaments 
achieve all possible values of $t(C_3,T)$ in the limit as $n$ tends to infinity
and their $4$-vertex subtournaments
are either transitive or contain a $4$-cycle (see~\cite[Proof of Observation 2.1]{LinM16});
therefore, they show that the upper boundary in Figure~\ref{fig-dens34} is asymptotically feasible.  

\section{Preliminaries}
\label{sec-prelim}

In this section, we introduce the notation used throughout the paper.
The set of all positive integers is denoted by $\NN$ and
the set of integers $1,\ldots,n$ by $[n]$.
Some of the matrices that we consider have complex eigenvalues and
the complex unit will be denoted by $\unit$.
If $A$ is a matrix (or a vector), then we write $A^T$ for its transpose and $A^*$ for its conjugate transpose;
in particular, if $A$ is real, then $A^T=A^*$.
The \emph{trace} of a square matrix $A$ is the sum of the entries in its diagonal and
is denoted by $\Tr A$.
We let $\left\langle\cdot\mid\cdot\right\rangle$ denote the standard inner (dot) product on $\RR^n$. 
We use $\JJ_n$ to denote the square matrix of order $n$ such that each entry of $\JJ_n$ is equal to one;
if $n$ is clear from the context, we will omit the subscript.
Note that $\JJ_n$ has one eigenvalue equal to $n$ and the remaining $n-1$ eigenvalues are equal to zero.
The $n$-dimensional column vector with all entries equal to one is denoted by $\jj_n$ and
we again omit the subscript when $n$ is clear from the context.
Note that $\JJ_n=\jj_n\jj_n^T$.

\subsection{Tournament matrices}

We say that a square matrix $A$ of order $n$ is a \emph{tournament matrix}
if $A$ is non-negative and $A+A^T=\JJ$;
in particular, if $A$ is a tournament matrix, then each diagonal entry of $A$ is equal to $1/2$.
Every $n$-vertex tournament $T$ can be associated with a tournament matrix $A$ of order $n$,
which we refer to as the \emph{adjacency matrix} of $T$, in the following way.
Each diagonal entry $A$ is equal to $1/2$ and, 
for $i\neq j$, the entry of $A$ in the $i$-th row and the $j$-th column (denoted $A_{i,j}$) is equal to $1$
if $T$ contains an arc from the $i$-th vertex to the $j$-th vertex, and
it is equal to $0$  otherwise.
The following proposition readily follows.

\begin{proposition}
\label{prop-trace}
Let $T$ be a tournament on $n$ vertices, $A$ be the adjacency matrix of $T$ and $\ell\ge 3$.
The number of homomorphisms of $C_\ell$ to $T$ is $\Tr A^\ell+O(n^{\ell-1})$. 
\end{proposition}

Recall that the trace of a matrix is equal to the sum of its 
eigenvalues and that the eigenvalues of the $\ell$-th power 
of a matrix are the $\ell$-th powers of its eigenvalues.
In view of Proposition~\ref{prop-trace}, for $\ell\geq1$, we define 
$\sigma_\ell(A)$ for a square matrix $A$ of order $n$ to be
\[\sigma_\ell(A)=\frac{1}{n^\ell}\sum_{i=1}^n \lambda_i^\ell = \frac{1}{n^\ell}\Tr A^\ell\]
where $\lambda_1,\ldots,\lambda_n\in\mathbb{C}$ are the eigenvalues of $A$.
Note that the normalization of the sum is chosen in such a way that $\sigma_1(A)=1/2$ for every tournament matrix~$A$.

Next, we argue that Conjecture~\ref{conj-main} is equivalent to the following.
\begin{conjecture}
\label{conj-matrix}
If $A$ is a tournament matrix, then $\sigma_4(A)\ge g(\sigma_3(A))$.
\end{conjecture}
Indeed, Conjecture~\ref{conj-matrix} implies Conjecture~\ref{conj-main} by Proposition~\ref{prop-trace}.
In the other direction, suppose that there exists a tournament matrix $A$ of order $n$ such that
$\sigma_4(A)<g(\sigma_3(A))$. We consider the following (random) tournament $T$ with $k\cdot n$ vertices, $k\in\NN$:
the vertices of $T$ are split into $n$ sets $V_1,\ldots,V_n$, each containing $k$ vertices, and
a vertex $v\in V_i$ is joined by an arc to a vertex $v'\in V_j$ with probability $A_{i,j}$;
note that $v'$ is joined by an arc to $v$ with probability $A_{j,i}=1-A_{i,j}$,
i.e., the tournament $T$ is well defined. Since $n$ is fixed, for $\ell\in\{3,4\}$ and large $k$, the number of 
homomorphisms from $C_\ell$ to $T$ is $\sigma_\ell(A)(nk)^\ell + O(k^{\ell-1})$ with high 
probability and so Conjecture~\ref{conj-main} fails for $t(C_3,T)\approx \sigma_3(A)$. 

We conclude this subsection by establishing the following lemma. A similar result appears in 
Brauer and Gentry~\cite{BraG68}, but for a slightly different definition of a tournament matrix.

\begin{lemma}
\label{lm-positive}
If $A$ is a tournament matrix, then every eigenvalue of $A$ has non-negative real part.
\end{lemma}

\begin{proof}
Let $\lambda$ be any eigenvalue of $A$ and let $v$ be a corresponding eigenvector.
Observe that the following holds:
\begin{align*}
0\le\overline{(\jj^T v)}(\jj^T v)&=v^*\JJ v=v^*(A+A^T)v=v^*(Av)+(v^*A^T)v\\
   &=v^*(\lambda+\overline{\lambda})v=(\lambda+\overline{\lambda})v^*v\,.
\end{align*}   
Since $v^*v$ is a non-negative real number, it follows that $\lambda+\overline{\lambda}$ is a non-negative real.
In particular, the real part of $\lambda$ is non-negative.
\end{proof}

\subsection{Tournament limits}

One of the substantial recent developments in combinatorics 
is the theory of graph limits which aims to provide analytic tools to represent and analyze large graphs. 
In an analogous way, one can develop a limit theory for tournaments,
in which essentially all of the foundational results for graph limits,
which can be found, e.g., in the monograph on graph limits by Lov\'asz~\cite{Lov12},
translate to similar statements for tournament limits with essentially the same proofs. Below, we define tournament
limits and outline some of the basic results that we will use.

A \emph{tournament limit} is a measurable function $W:[0,1]^2\to [0,1]$ such that $W(x,y)+W(y,x)=1$
for all $(x,y)\in[0,1]^2$. One can define the density of the cycle $C_{\ell}$ in $W$ as follows:
\[t(C_{\ell},W)=\int_{x_1,\ldots,x_\ell\in [0,1]}W(x_1,x_2)W(x_2,x_3)\cdots W(x_{\ell-1},x_{\ell})W(x_{\ell},x_1)\dd x_1\cdots x_{\ell}\,.\]
Note that any $n$-dimensional tournament matrix $A$ can be represented by a tournament 
limit $W_A$ by dividing $[0,1]$ into sets $I_1,\dots,I_n$ of measure 
$1/n$ and setting $W$ equal to $A_{i,j}$ on the set $I_i\times I_j$. 
It is easily observed that $t(C_\ell,W_A)$ is precisely $\sigma_\ell(A)$.
The following proposition links densities of cycles in tournament limits and in tournaments.
\begin{proposition}
\label{prop-limit}
The following two statements are equivalent for every sequence $(s_\ell)_{\ell\ge 3}$ of non-negative reals:
\begin{itemize}
\item There exists a tournament limit $W$ such that $t(C_{\ell},W)=s_{\ell}$ for every $\ell\ge 3$.
\item There exists a sequence $(T_i)_{i\in\NN}$ of tournaments with increasing orders such that
      \[\lim_{i\to\infty}t(C_{\ell},T_i)=s_{\ell}\]
      for every $\ell\ge 3$.
\end{itemize}
\end{proposition}
The first statement easily implies the second by letting $T_i$ be a \emph{$W$-random tournament} of order $i$;
that is, we let $x_1,\dots,x_i$ be independent uniformly random points of $[0,1]$ and join the $i$th vertex 
to the $j$th with probability $W(x_i,x_j)$. For the other direction, the tournament limit $W$ can be constructed
by adapting one of the existing arguments in the graph case,
e.g., the argument of Lov\'asz and Szegedy~\cite{LovS06} based on weak regularity and the Martingale Convergence Theorem.
In light of Proposition~\ref{prop-limit}, Conjecture~\ref{conj-main} is equivalent to the following.
\begin{conjecture}
\label{conj-limit}
For every tournament limit $W$, it holds that $t(C_4,W)\ge g(t(C_3,W))$.
\end{conjecture}

The notion of regularity decompositions of graphs readily extends to tournaments.
We present here the notion of weak regular partitions introduced by Frieze and Kannan in~\cite{FriK99}
adapted to the setting of tournament limits.
We use $|X|$ to denote the measure of a measurable subset $X$ of $[0,1]$.
Given a tournament limit $W$ and $\varepsilon\in(0,1)$, a partition $Z_1,\ldots,Z_n$ of $[0,1]$ into sets 
of measure $1/n$ is \emph{weak $\varepsilon$-regular} for $W$ if
\[\left|\int_{(x,y)\in X\times Y}W(x,y)\dd x\dd y-\sum_{i,j=1}^n A_{i,j}\cdot|Z_i\cap X|\cdot |Z_j\cap Y|\right|\le\varepsilon\]
for all measurable subsets $X$ and $Y$ of $[0,1]$,
where $A$ is the tournament matrix defined by
\[A_{i,j}=\frac{\int_{(x,y)\in Z_i\times Z_j}W(x,y)\dd x\dd y}{|Z_i|\cdot|Z_j|}\,.\]
We say that a tournament limit $W'$ is a \emph{weak $\varepsilon$-regular approximation} of $W$
if there exists a weak $\varepsilon$-regular partition $\{Z_1,\ldots,Z_n\}$ such that
$W'(x,y)=A_{i,j}$ for $(x,y)\in Z_i\times Z_j$, $i,j\in [n]$,
where $A$ is the tournament matrix associated with the partition.

The results of Frieze and Kannan~\cite{FriK99} adapted to the setting of tournament limits and
the corresponding arguments for graph limits~\cite{LovS06} yield the following:
for every tournament limit $W$ and $k\geq2$, there exists a weak $1/k$-regular partition $\{Z_{k,1},\ldots,Z_{k,n_k}\}$
with the following properties: (a) $n_k$ is bounded by a function of $k$, and (b) the partitions are \emph{refining} 
in the sense that, for every $k<k'$ and $i'\in [n_{k'}]$,
the set $Z_{k',i'}$ is contained in $Z_{k,i}$ for some $i\in [n_k]$.
It can be shown analogously to the graph case that
the corresponding weak $1/k$-regular approximations converge to $W$ in $L_1$.
In particular, it holds that
\[\lim_{k\to\infty}\sigma_{\ell}(A_k)=t(C_{\ell},W)\]
for every $\ell\ge 3$, 
where, for $k\in\NN$, $A_k$ is the tournament matrix associated with the partition $Z_{k,1},\ldots,Z_{k,n_k}$.

We conclude with a proposition on the density of $C_3$
in a weak regular approximation of a tournament limit.
The proof of the proposition is also valid in a more general setting of step approximations of tournament limits,
which need not be weak regular, however,
we prefer stating the proposition in the restricted setting of weak regular approximations
to avoid introducing additional notation not needed for our exposition.

\begin{proposition}
\label{prop-WT3}
Let $W$ be a tournament limit and $W'$ a weak $\varepsilon$-regular approximation of $W$.
It holds that $t(C_3,W)\le t(C_3,W')$.
\end{proposition}

\begin{proof}
We begin by showing that any tournament limit $U$ satisfies
\begin{equation}
t(C_3,U)=\frac{1}{2}-\frac{3}{2}\int_{x\in [0,1]}\left(\int_{y\in [0,1]}U(x,y)\dd y\right)^2\dd x\,.
\label{eq-UT3}
\end{equation}
To do this, we derive two identities based on the symmetry of variables $x$, $y$ and $z$. Firstly,
\[ 1 = \int_{x,y,z\in [0,1]}(U(x,y)+U(y,x))(U(x,z)+U(z,x))(U(y,z)+U(z,y))\dd x\dd y\dd z \]
\[ = 2\int_{x,y,z\in [0,1]} U(x,y)U(y,z)U(z,x)\dd x\dd y\dd z\]
\[ +\ 6\int_{x,y,z\in [0,1]} U(x,y)U(y,z)U(x,z)\dd x\dd y\dd z 
\]
\begin{equation}
  = 2t(C_3,U)+6\int_{x,y,z\in [0,1]} U(x,y)U(y,z)U(x,z)\dd x\dd y\dd z \,, \label{eq-UT3a}
\end{equation}
and similarly,
\[ \int_{x\in [0,1]}\left(\int_{y\in [0,1]}U(x,y)\dd y\right)^2\dd x = \int_{x,y,z\in [0,1]}U(x,y)U(x,z)\dd x\dd y\dd z \]
\[ = \int_{x,y,z\in [0,1]}U(x,y)U(x,z)(U(y,z)+U(z,y))\dd x\dd y\dd z \]
\begin{equation}
  = 2\int_{x,y,z\in [0,1]} U(x,y)U(x,z)U(y,z)\dd x\dd y\dd z \,. \label{eq-UT3b}
\end{equation}
Noticing that the integrals on the last lines in \eqref{eq-UT3a} and \eqref{eq-UT3b} are the same, the equality \eqref{eq-UT3} is obtained.
Hence, the inequality from the statement of the proposition is equivalent to
\begin{equation}
   \int_{x\in [0,1]} f'(x)^2\dd x\le
   \int_{x\in [0,1]} f(x)^2\dd x\,,\label{eq-WT3}
\end{equation}
where for brevity we have set
\[ f'(x)=\int_{y\in [0,1]}W'(x,y)\dd y \qquad\mbox{and}\qquad f(x)=\int_{y\in [0,1]}W(x,y)\dd y\,.\]
Since
\[f'(x)=\frac{1}{|Z_i|}\int_{x'\in Z_i}f(x')\dd x'\]
for every $x$ in a part $Z_i$ of the weak $\varepsilon$-regular partition defining the tournament limit $W'$, it holds that
\begin{align}
\int_{x\in Z_i} f'(x)^2\dd x &=|Z_i|\left(\frac{1}{|Z_i|}\int_{x'\in Z_i}f(x')\dd x'\right)^2 \nonumber \\
&\le\int_{x\in Z_i} f(x)^2\dd x \,,
\label{eq-WT3Z}
\end{align}
where the last line follows from the Cauchy--Schwarz inequality. 
Summing the inequalities obtained from applying \eqref{eq-WT3Z} to each $Z_i$ yields \eqref{eq-WT3}.
\end{proof}

\section{Regime of two parts}
\label{sec-reg2}

Our goal in this section is to prove Conjecture~\ref{conj-matrix} in the case that $\sigma_3(A)\ge 1/32$, as well as describe the 
tournament matrices which achieve equality. We then apply this result to characterise the extremal tournament 
limits for Conjecture~\ref{conj-limit} for $t(C_3,W)\geq 1/32$. Throughout the proof of the next theorem, we will frequently 
use the property that the trace of a product of matrices is invariant under ``cyclic permutations'', i.e., 
$\Tr\left(M_1M_2\cdots M_k\right)=\Tr\left(M_2 \cdots M_kM_1\right)$. 

\begin{theorem}
\label{thm-reg2}
Let $A$ be a tournament matrix of order $n$.
If $\sigma_3(A)\ge 1/32$, then $\sigma_4(A)\ge g(\sigma_3(A))$ and
equality holds if and only if there exists a vector $z\in\RR^n$ such that
$A_{i,j}=1/2+z_i-z_j$ for $i,j\in [n]$.
\end{theorem}

\begin{proof}
Fix a tournament matrix $A$ of order $n$.
Let $B=\JJ-2A$. Note that $B$ is a skew-symmetric matrix, i.e., $B=-B^T$.
It follows (see, e.g.,~\cite[p. 293]{Gan98}) that $B$ can be written as
$B=ULU^T$ where the columns $v_1,v_2,\dots,v_n$ of $U$ form an orthonormal 
basis of $\RR^n$ and $L$ has the form
\[L=
  \begin{bmatrix}
  0 &  \lambda_1n & 0 & 0 & \cdots & 0 & 0 \\
  -\lambda_1n & 0 & 0 & 0 & \cdots & 0 & 0 \\
  0 & 0 & 0 &  \lambda_2n & \cdots & 0 & 0 \\
  0 & 0 & -\lambda_2n & 0 & \cdots & 0 & 0 \\
  \vdots & \vdots & \vdots & \vdots & \ddots & \vdots & \vdots \\
  0 & 0 & 0 & 0 & \cdots & 0 & \lambda_kn \\
  0 & 0 & 0 & 0 & \cdots & -\lambda_kn & 0
  \end{bmatrix}\]
if $n$ is even, and
\[L=
  \begin{bmatrix}
  0 &  \lambda_1n & 0 & 0 & \cdots & 0 & 0 & 0 \\
  -\lambda_1n & 0 & 0 & 0 & \cdots & 0 & 0 & 0 \\
  0 & 0 & 0 &  \lambda_2n & \cdots & 0 & 0 & 0 \\
  0 & 0 & -\lambda_2n & 0 & \cdots & 0 & 0 & 0 \\
  \vdots & \vdots & \vdots & \vdots & \ddots & \vdots & \vdots & \vdots \\
  0 & 0 & 0 & 0 & \cdots & 0 &  \lambda_kn & 0 \\
  0 & 0 & 0 & 0 & \cdots & -\lambda_kn & 0 & 0 \\
  0 & 0 & 0 & 0 & \cdots & 0 & 0 & 0
  \end{bmatrix}\]
otherwise, where $k=\lfloor n/2\rfloor$ and $\lambda_1,\ldots,\lambda_{k}$ are real numbers. (Note that they are not the eigenvalues of $B$.)
Since replacing $v_{2i-1}$ and $v_{2i}$ with
$v_{2i-1} \cos\beta +v_{2i} \sin\beta $ and $v_{2i} \cos\beta -v_{2i-1} \sin\beta $, respectively,
does not change the matrix $B$ (this corresponds to rotating the basis inside the plane
spanned by $v_{2i-1}$ and $v_{2i}$),
we can assume that the vectors $v_2,v_4,\ldots,v_{2k}$ are orthogonal to the vector $\jj$.
Set $\alpha_i=\cos^{-1}\left\langle v_{2i-1}\mid n^{-1/2}\jj\right\rangle$ for $i\in [k]$, and
additionally set $\alpha_{k+1}=\cos^{-1}\left\langle v_{2k+1}\mid n^{-1/2}\jj\right\rangle$ if $n$ is odd. 

We next examine $\Tr A^3$ and $\Tr A^4$ in terms of $\JJ$ and $B$.
We start with the trace of $A^3$:
\[ 8 \Tr A^3 = \Tr (\JJ-B)^3 = \Tr\JJ^3 - 3\Tr\JJ^2B + 3\Tr\JJ B^2 - \Tr B^3\,.\]
Since both $B$ and $B^3$ are skew-symmetric, it follows that $\Tr\JJ^2 B=0$ and $\Tr B^3=0$.
We next analyze the term $\Tr\JJ B^2$. Since $v_1,\dots,v_n$ are mutually orthogonal and $v_2,v_4,\dots,v_{2k}$ are orthogonal to $\vec{j}$, we have
\[ \Tr\JJ B^2 = \frac{1}{n} \Tr\JJ^2 B^2=\frac{1}{n} \Tr\JJ B^2\JJ = 
\frac{1}{n}\Tr\JJ(ULU^T)^2\JJ = \frac{1}{n}\Tr\JJ U L^2U^T \JJ\]
\[=-n^2\sum_{i=1}^k\lambda_i^2\left\langle v_{2i-1} \mid \jj\right\rangle^2
=-n^3\sum_{i=1}^k\lambda_i^2\cos^2\alpha_i\,.\]
Hence, we obtain that
\begin{equation}
8\sigma_3(A)=1-3\sum_{i=1}^k\lambda_i^2\cos^2\alpha_i.
\label{eq-sigma3}
\end{equation}
Similarly, we can express the trace of $A^4$ as follows:
\[ 16 \Tr A^4= \Tr\JJ^4 - 4\Tr\JJ^3B + 4\Tr \JJ^2 B^2 +2\Tr \JJ B\JJ B-4\Tr\JJ B^3+\Tr B^4\,.\]
Since $B$ and $B^3$ are skew-symmetric, it follows that $\Tr\JJ^3B=0$, $\Tr \JJ B\JJ B=0$ and $\Tr\JJ B^3=0$.
Also, $\Tr B^4=2n^4\sum_{i=1}^k\lambda_i^4$ by the cyclic property.
Consequently, we get that
\begin{equation}
16\sigma_4(A)=1-4\sum_{i=1}^k\lambda_i^2\cos^2\alpha_i+2\sum_{i=1}^k\lambda_i^4.
\label{eq-sigma4}
\end{equation}
Recall that, if $\sigma_3(A)\in [1/32,1/8]$, then $g(\sigma_3(A))=\frac{1}{16}(z^4 + (1-z)^4)$ where $z\in[1/2,1]$ 
such that $\sigma_3(A)=\frac{1}{8}(z^3+(1-z)^3)$. 
So, for $\sigma_3(A)$ in the considered range, we have \[8\sigma_3(A)=z^3+(1-z)^3=1-3(z-z^2)\,.\]
By comparing this equation to \eqref{eq-sigma3}, it must be the case that
\begin{equation}
z-z^2=\sum_{i=1}^k\lambda_i^2\cos^2\alpha_i \,.
\label{eq-sigmaz}
\end{equation}
It follows that
\begin{align*}
16g(\sigma_3(A)) & =  z^4 + (1-z)^4=1-4z+6z^2-4z^3+2z^4 \\
                 & =  1-4(z-z^2)+2(z-z^2)^2 \\
		 & =  1-4\sum_{i=1}^k\lambda_i^2\cos^2\alpha_i+2\left(\sum_{i=1}^k\lambda_i^2\cos^2\alpha_i\right)^2\,.
\end{align*}
Combining this with \eqref{eq-sigma4}, we see that
the inequality $\sigma_4(A)\ge g(\sigma_3(A))$ holds if and only if
\begin{equation}
\sum_{i=1}^k\lambda_i^4\ge\left(\sum_{i=1}^k\lambda_i^2\cos^2\alpha_i\right)^2 \,,
\label{eq-sigma34}
\end{equation}
and $\sigma_4(A)=g(\sigma_3(A))$ if and only if \eqref{eq-sigma34} holds with equality.

Since $v_1,\ldots,v_n$ form an orthonormal basis of $\RR^n$, 
\[\sum_{i=1}^n\left\langle v_i\mid n^{-1/2}\jj\right\rangle^2
=\left\langle n^{-1/2}\jj\mid n^{-1/2}\jj\right\rangle^2= 1.\]
Thus, $\sum_{i=1}^k\cos^2\alpha_i= 1$ if $n$ is even and $\sum_{i=1}^{k+1}\cos^2\alpha_i =1$ otherwise.
In either case, $\sum_{i=1}^k\cos^4\alpha_i\le 1$ and the equality holds if and only if
exactly one of the values of $\alpha_1,\ldots,\alpha_k$ is equal to zero and the remainder are equal to $\pi/2$.
Since the Cauchy--Schwarz inequality implies that
\begin{equation}
\left(\sum_{i=1}^k\lambda_i^2\cos^2\alpha_i\right)^2\le
\left(\sum_{i=1}^k\lambda_i^4\right)\cdot\left(\sum_{i=1}^k\cos^4\alpha_i\right)\,,
\label{eq-sigmaCS}
\end{equation}
we obtain that the inequality \eqref{eq-sigma34} indeed holds.

Now, assume that the inequality \eqref{eq-sigma34} holds with equality.
As we have seen, this can only occur if exactly one of the $\alpha_i$ are 
zero and the rest are $\pi/2$.
By symmetry, we can assume that $\alpha_1=0$ and $\alpha_i=\pi/2$ for $i>1$. It follows that $\lambda_2=\cdots=\lambda_k=0$, and $v_{1}=n^{-1/2}\jj$.
Hence, as $B=ULU^T$, the entry $B_{i,j}$ is equal to $\lambda_1 n^{1/2}(v_{2,j}-v_{2,i})$ 
for all $i,j\in [n]$. It follows that, for $\sigma_3(A)\in [1/32,1/8]$, if 
$\sigma_4(A)=g(\sigma_3(A))$, then $A_{i,j}=1/2+z_i-z_j$ where $z_i=\lambda_1 n^{1/2}v_{2,i}/2$.
Conversely, any matrix of this form satisfies \eqref{eq-sigma34} with equality and therefore
satisfies $\sigma_4(A)=g(\sigma_3(A))$.
\end{proof}

Reinterpreting Theorem~\ref{thm-reg2} in the language of tournament limits, we obtain the following corollary.
\begin{corollary}
\label{cor-reg2}
Let $W$ be a tournament limit.
If $t(C_3,W)\ge 1/32$, then $t(C_4,W)\ge g(t(C_3,W))$ and
the equality holds if and only if
there exists a measurable function $f:[0,1]\to [0,1/2]$ such that
$W(x,y)=1/2+f(x)-f(y)$ for almost every $(x,y)\in [0,1]^2$.
\end{corollary}

\begin{proof}
Let $(W_k)_{k\in\NN}$ be a sequence of refining weak $1/k$-regular approximations of $W$ and let
$A_k$, $k\in\NN$, be the corresponding tournament matrices.
Since $t(C_3,W)\ge 1/32$, 
it holds that $\sigma_3(A_k)=t(C_3,W_k)\ge 1/32$ by Proposition~\ref{prop-WT3}.
Thus, by Theorem~\ref{thm-reg2},
we have that 
\[t(C_4,W_k)=\sigma_4(A_k)\geq g\left(\sigma_3(A_k)\right)=g(t(C_3,W_k))\, .\]
and so $t(C_4,W)\geq g(t(C_3,W))$ by the fact that $(W_k)_{k\in\NN}$ converges to $W$ in $L_1$. 

To prove the structure of $W$ in the case of equality, assume that $t(C_4,W)=g(t(C_3,W))$. Let $n_k$ be the order of $A_k$ for $k\in\NN$, and
let $\lambda_{k,1},\ldots,\lambda_{k,\lfloor n_k/2\rfloor}$, $\alpha_{k,1},\ldots,\alpha_{k,\lfloor n_k/2\rfloor}$ and $B_k$
be defined as in the proof of Theorem~\ref{thm-reg2} 
(we may assume that $n_k$ is even and so $\alpha_{k,\lfloor n_k/2\rfloor+1}$ is not defined). The analysis of the case of equality in the proof of Theorem~\ref{thm-reg2} implies that
\begin{equation}
  \lim_{k\to\infty}\lambda_{k,1}=\sqrt{\frac{1-8t(C_3,W)}{3}}
  \qquad\mbox{and}\qquad
  \lim_{k\to\infty}\alpha_{k,1}=0\,.
  \label{eq-reg2-lim}
\end{equation}
Let $\{Z_{k,1},\ldots,Z_{k,n_k}\}$ be the partition of the interval $[0,1]$ corresponding to $W_k$, and let $w_k=(w_{k,1},\ldots, w_{k,n_k})$ be the vector $v_2$ as defined in the proof of Theorem~\ref{thm-reg2}. Define a function $f_k:[0,1]\to\RR$ by setting
\[f_k(x)=\frac{\lambda_{k,1}n_k^{1/2}}{2}\left(w_{k,i}-\frac{\sum_{i'=1}^{n_k}w_{k,i'}}{n_k}\right)\]
where $i\in[n_k]$ such that $x\in Z_{k,i}$.
It follows from \eqref{eq-reg2-lim} that
\begin{equation}
\lim_{k\to\infty}\int_{x,y\in [0,1]}
		 \left|W_k(x,y)-\left(1/2+f_k(x)-f_k(y)\right)\right|\dd x\dd y
		 =0\,.
\label{eq-reg2-L1}
\end{equation}
Also observe that the definition of $f_k$ implies that
\begin{equation}
\int_{x\in [0,1]}f_k(x)=0\,.
\label{eq-reg2-zero1}
\end{equation}
We next define functions $\ff_k:[0,1]\to\RR$ by setting
\[\ff_k(x)=\frac{1}{|Z_{k,i}|}\int_{(x',y)\in Z_{k,i}\times [0,1]}W(x',y)\dd x'\dd y-\frac{1}{2}\]
for $x\in Z_{k,i}$, $i\in [n_k]$.
Note that
\begin{equation}
\ff_k(x)=\int_{y\in [0,1]} W_k(x,y)\dd y-\frac{1}{2}\,.
\label{eq-reg2-Wk}
\end{equation}
In particular, $\ff_k(x)\in [-1/2,1/2]$ for all $x\in [0,1]$.
Observe that the just defined functions satisfy that
\[\int_{x\in Z_{k,i}}\ff_{k'}(x)\dd x=\int_{x\in Z_{k,i}}\ff_k(x)\dd x\]
for every $k\in\NN$, $i\in [n_k]$ and $k'\ge k$.
In particular, it holds that
the sequence $(\ff_k)_{k\in\NN}$ forms a martingale when viewed as a sequence of random variables on $[0,1]$.
So, Doob's Martingale Convergence Theorem yields that
the sequence $(\ff_k)_{k\in\NN}$ $L_1$-converges to a function $\ff:[0,1]\to [-1/2,1/2]$.

We derive by applying the $L_1$-convergence of $(\ff_k)_{k\in\NN}$,
\eqref{eq-reg2-Wk}, \eqref{eq-reg2-zero1}, the triangle inequality and \eqref{eq-reg2-L1} (in this order) that
\[\lim_{k\to\infty}\int_{x\in [0,1]}\left|\ff(x)-f_k(x)\right|\dd x=
  \lim_{k\to\infty}\int_{x\in [0,1]}\left|\ff_k(x)-f_k(x)\right|\dd x\]
\[=\lim_{k\to\infty}\int_{x\in [0,1]}\left|\int_{y\in [0,1]} W_k(x,y)\dd y-1/2-f_k(x)\right|\dd x\]
\[=\lim_{k\to\infty}\int_{x\in [0,1]}\left|\int_{y\in [0,1]} W_k(x,y)-\left(1/2+f_k(x)-f_k(y)\right)\dd y\right|\dd x\]
\[\le\lim_{k\to\infty}\int_{x,y\in [0,1]}\left|W_k(x,y)-\left(1/2+f_k(x)-f_k(y)\right)\right|\dd x\dd y=0\,.\]
This implies that the sequence $(f_k)_{k\in\NN}$ also $L_1$-converges to the function $\ff$.
It follows from \eqref{eq-reg2-L1} and the $L_1$-convergence of $W_k$ to $W$
that $W(x,y)$ is equal to $1/2+\ff(x)-\ff(y)$ for almost every $(x,y)\in [0,1]^2$.

It remains to shift $\ff$ so that its range lies in $[0,1/2]$. Let $z_0$ be the infimum of those values $z$ such that the measure of $\ff^{-1}\left((-\infty,z]\right)$ is positive, and
define a function $f:[0,1]\to [0,1/2]$ as follows:
\[
f(x)=\begin{cases}
     \ff(x)-z_0 & \mbox{if $\ff(x)\in [z_0,z_0+1/2]$, and} \\
     0 & \mbox{otherwise.}
     \end{cases}
\]
Since $W(x,y)=1/2+\ff(x)-\ff(y)$ for almost every $(x,y)\in [0,1]^2$ and
$W(x,y)\in [0,1]$ for all $(x,y)\in [0,1]^2$,
the set of $x\in [0,1]$ such that the second case in the definition of $f(x)$ applies has measure zero.
It follows that $W(x,y)=1/2+f(x)-f(y)$ for almost every $(x,y)\in [0,1]^2$ as desired.
\end{proof}

\section{Regime of three parts}
\label{sec-reg23}
Having confirmed Conjecture~\ref{conj-main} in the regime of two parts, we now turn towards the next case, namely $1/72 \leq \sigma_3(A)\leq 1/32$ (Theorem~\ref{thm-reg23}). Indeed, the proof of Theorem~\ref{thm-reg23} will apply to both regimes, although it does not characterise the extremal tournaments.

We start with analyzing the following optimization problem, which we refer to as the problem \spectrum{}.
This optimization problem is obtained from constraints that (normalized) eigenvalues of
a non-negative matrix of order $n$ with trace $n/2$ must satisfy.
We state this formally in Lemma~\ref{lm-problem}, which follows the statement of the problem.

\begin{center}
\begin{tabular}{ll}
& {\bf Optimization problem \spectrum{}} \\
\hline
Parameters: & reals $s_3\in [0,1/8]$ and $\rho\in [0,1/2]$\\
            & non-negative integers $k$ and $\ell$ such that $k+\ell\ge 1$ \\
Variables: & real numbers $r_1,\ldots,r_{k}$, $a_1,\ldots,a_{\ell}$ and $b_1,\ldots,b_{\ell}$ \\
Constraints: & $0\le r_1,\ldots,r_{k}\le\rho$\\
             & $0\le a_1,\ldots,a_{\ell}$\\
	     & $\rho+\sum\limits_{i=1}^{k}r_i+2\sum\limits_{i=1}^{\ell}a_i=1/2$\\
	     & $\rho^3+\sum\limits_{i=1}^{k}r_i^3+2\sum\limits_{i=1}^{\ell}\left(a_i^3-3a_ib_i^2\right)=s_3$\\
Objective: & minimize $\rho^4+\sum\limits_{i=1}^{k}r_i^4+2\sum\limits_{i=1}^{\ell}\left(a_i^4-6a_i^2b_i^2+b_i^4\right)$
\end{tabular}
\end{center}

\begin{lemma}
\label{lm-problem}
Let $A$ be a tournament matrix of order $n$ with spectral radius equal to $\rho\cdot n$.
Let $k$ be one less than the number of real eigenvalues of $A$ (counting multiplicities) and
$\ell$ the number of conjugate pairs of complex eigenvalues (again counting multiplicities).
Further, let $\rho\cdot n,r_1\cdot n,\ldots,r_{k}\cdot n$ be the $k+1$ real eigenvalues and
$(a_1\pm\unit b_1)n,\ldots,(a_{\ell}\pm\unit b_{\ell})n$ be the $\ell$ pairs of complex eigenvalues.
Then the numbers $r_1,\ldots,r_{k}$, $a_1,\ldots,a_{\ell}$ and $b_1,\ldots,b_{\ell}$
satisfy all constraints in the optimization problem \spectrum{}
for the parameters $s_3=\sigma_3(A)$, $\rho$, $k$ and $\ell$.
\end{lemma}

\begin{proof}
Since $\rho\cdot n$ is the spectral radius of $A$,
it holds that $r_i\le\rho$ for every $i\in [k]$ and that
$\rho\cdot n$ is an eigenvalue of $A$ by the Perron--Frobenius theorem.
Since the real part of every eigenvalue of $A$ is non-negative by Lemma~\ref{lm-positive},
all $r_1,\ldots,r_{k}$ and $a_1,\ldots,a_{\ell}$ are non-negative.
Since the diagonal entries of $A$ are all $1/2$ and trace of $A$ is equal to
the sum of its eigenvalues, we have
\[\rho n+\sum\limits_{i=1}^{k}r_i n+2\sum\limits_{i=1}^{\ell}a_i n = n/2.\]
Similarly, the trace of $A^3$ gives us
\[s_3n^3=\sigma_3(A)n^3 = \rho^3n^3+\sum\limits_{i=1}^{k}r_i^3n^3+2\sum\limits_{i=1}^{\ell}\left(a_i^3-3a_ib_i^2\right)n^3.\]
Thus, we conclude that the numbers $r_1,\ldots,r_{k}$, $a_1,\ldots,a_{\ell}$ and $b_1,\ldots,b_{\ell}$
satisfy all constraints in the optimization problem \spectrum{}.
\end{proof}

Note that the objective function of \spectrum{} is precisely $\sigma_4(A)$. 
Next, we analyze the structure of optimal solutions of the optimization problem \spectrum{}.

\begin{lemma}
\label{lm-lagrange}
Let $r_1,\ldots,r_{k}$, $a_1,\ldots,a_{\ell}$ and $b_1,\ldots,b_{\ell}$ be an optimal solution
of the optimization problem \spectrum{} with the parameters $s_3$, $\rho$, $k$ and $\ell$.
Then, at least one of the following two cases must hold:
\begin{itemize}
\item There exist positive reals $r'$ and $r''$ such that $r_1,\ldots,r_{k}\in\{0,r',r'',\rho\}$,
      $(a_1,b_1),\ldots,(a_{\ell},b_{\ell})\in\{(0,0),(r',0),(r'',0)\}$.
\item There exist reals $a'$ and $b'\neq 0$ such that $r_1,\ldots,r_{k}\in\{0,\rho\}$ and
      $(a_1,b_1),\ldots$, $(a_{\ell},b_{\ell})\in\{(0,0),(a',b'),(a',-b')\}$.
\end{itemize}
\end{lemma}

\begin{proof}
The method of Lagrange multipliers implies that 
the gradient of the objective function is a linear combination of
the gradient of the two equality constraints when restricted to the entries indexed
by $r_i\notin\{0,\rho\}$, by $a_i\neq 0$ and $b_i\neq 0$, i.e., when we are not on the boundary of the feasible set.
In particular,
the rank of the matrix $M$ with rows being the described restrictions of the three gradient vectors is at most two.

We first analyze the case that one of the numbers $b_1,\ldots,b_{\ell}$ is non-zero.
Our aim is to show that the second case described in the statement of the lemma applies.
By symmetry, we can assume that $b_1\neq 0$.
Also note the following holds for every $i\in [\ell]$: if $b_i\not=0$, then $a_i\not=0$.
Indeed, if $a_i=0$ and $b_i\neq 0$,
then setting $b_i=0$ does not affect the constraints and decreases the objective function,
which contradicts that the solution is optimal.
It follows that $a_1$ is positive.

Suppose that there exist $r_i$ such that $0<r_i<\rho$.
The restriction of the matrix $M$ to the columns corresponding to $a_1$, $b_1$ and $r_i$ is the following.
\begin{equation}
  \begin{bmatrix}
  2 & 0 & 1 \\
  6a_1^2-6b_1^2 & -12a_1b_1 & 3r_i^2 \\
  8a_1^3-24a_1b_1^2 & 8b_1^3-24a_1^2b_1 & 4r_i^3
  \end{bmatrix}
  \label{eq-lagrange}
\end{equation}
Dividing the first column by $2$ and the second by $4b_1$,
dividing the second row by $3$ and the third by $2$,
and subtracting the last column from the first one yields
the following matrix, which has the same rank.
\[\begin{bmatrix}
  0 & 0 & 1 \\
  a_1^2-b_1^2-r_i^2 & -a_1 & r_i^2 \\
  2(a_1^3-3a_1b_1^2-r_i^3) & b_1^2-3a_1^2 & 2r_i^3
  \end{bmatrix}\]
This matrix is not full rank if and only if the determinant of its submatrix formed by the intersection of the
second and third rows with the first and second columns, which is equal to
\[-(a_1^2+b_1^2)^2+(3a_1^2-b_1^2)r_i^2-2a_1r_i^3\,,\]
is zero.
However, this determinant can be rewritten as
\[-(2a_1r_i+a_1^2+b_1^2)\left((r_i-a_1)^2+b_1^2\right)\,,\]
which is negative since $a_1>0$ and $b_1\neq0$.
It follows that $r_i\in\{0,\rho\}$ for all $i\in [k]$.

Further suppose that there exists an index $i\in[\ell]$ for which it holds $(a_i,b_i)\notin\{(0,0),(a_1,b_1),(a_1,-b_1)\}$.
If $b_i=0$, then the restriction of the matrix $M$ to the columns corresponding to $a_1$, $b_1$ and $a_i$
is the same as the restriction of the matrix $M$ considered in the previous paragraph with $r_i$ replaced by $a_i$ and
the corresponding column multiplied by two. In particular, the restriction cannot have rank two in this case.
Hence, we can assume that $b_i\neq 0$ and so $a_i>0$ (the argument is the same as when we argued that $a_1>0$).
The restriction of the matrix $M$ to the columns corresponding to $a_1$, $b_1$, $a_i$ and $b_i$ is the following.
\[\begin{bmatrix}
  2 & 0 & 2 & 0 \\
  6a_1^2-6b_1^2 & -12a_1b_1 & 6a_i^2-6b_i^2 & -12a_ib_i \\
  8a_1^3-24a_1b_1^2 & 8b_1^3-24a_1^2b_1 & 8a_i^3-24a_ib_i^2 & 8b_i^3-24a_i^2b_i
  \end{bmatrix}\]
The rank of this matrix is the same as the rank of the following matrix (the rows
are multiplied by $1/2$, $1/6$ and $1/4$,
the columns by $1$, $-a_1/2b_1$, $1$ and $-a_i/2b_i$, respectively).
\[\begin{bmatrix}
  1 & 0 & 1 & 0 \\
  a_1^2-b_1^2 & a_1^2 & a_i^2-b_i^2 & a_i^2 \\
  2a_1^3-6a_1b_1^2 & 3a_1^3-a_1b_1^2 & 2a_i^3-6a_ib_i^2 & 3a_i^3-a_ib_i^2
  \end{bmatrix}\]
By subtracting twice the second column from the first column and
twice the fourth column from the third column,
we obtain the following matrix.
\[\begin{bmatrix}
  1 & 0 & 1 & 0 \\
  -(a_1^2+b_1^2) & a_1^2 & -(a_i^2+b_i^2) & a_i^2 \\
  -4a_1(a_1^2+b_1^2) & 3a_1^3-a_1b_1^2 & -4a_i(a_i^2+b_i^2) & 3a_i^3-a_ib_i^2
  \end{bmatrix}\]
Since the last row of the matrix is a linear combination of the previous two rows (the operation that
we have performed so far has preserved this property of the matrix $M$),
it follows that
\begin{equation}
\frac{3a_1^2-b_1^2}{a_1}=\frac{3a_i^2-b_i^2}{a_i}
\label{eq-fraction}
\end{equation}
We now subtract the second row multiplied by the value given in \eqref{eq-fraction} from the third row and obtain the following matrix,
which has the rank two.
\[\begin{bmatrix}
  1 & 0 & 1 & 0 \\
  -(a_1^2+b_1^2) & a_1^2 & -(a_i^2+b_i^2) & a_i^2 \\
  -(a_1^2+b_1^2)^2/a_1 & 0 & -(a_i^2+b_i^2)^2/a_i & 0
  \end{bmatrix}\]
It follows that
\[\frac{(a_1^2+b_1^2)^2}{a_1}=\frac{(a_i^2+b_i^2)^2}{a_i}\,,\]
which yields that
\begin{equation}
b_i^2=-a_i^2+\sqrt{\frac{a_i}{a_1}}(a_1^2+b_1^2)\,.
\label{eq-bi-1}
\end{equation}
On the other hand, we derive from \eqref{eq-fraction} that
\begin{equation}
b_i^2=3a_i^2-\frac{a_i}{a_1}\left(3a_1^2-b_1^2\right)\,.
\label{eq-bi-2}
\end{equation}
We obtain by comparing \eqref{eq-bi-1} and \eqref{eq-bi-2} the following.
\begin{align*}
0 & =  4a_i^2-\frac{a_i}{a_1}\left(3a_1^2-b_1^2\right)-\sqrt{\frac{a_i}{a_1}}(a_1^2+b_1^2) \\
  & =  \sqrt{a_i}(\sqrt{a_i}-\sqrt{a_1})\left(4a_i+4\sqrt{a_1a_i}+\frac{a_1^2+b_1^2}{a_1}\right)
\end{align*}
Since both $a_1$ and $a_i$ are positive, this expression can be equal to zero only if $a_i=a_1$.
Consequently, the equality \eqref{eq-bi-1} implies that $b_i=b_1$ or $b_i=-b_1$,
which contradicts the choice of $(a_i,b_i)$.
Hence, we have established that if at least one of $b_1,\ldots,b_{\ell}$ is non-zero,
then the second case indeed applies.

We now consider the case that $b_1=\cdots=b_{\ell}=0$.
Suppose that the first case in the statement of the lemma does not apply.
This implies that there exist three distinct positive reals $\alpha$, $\beta$ and $\gamma$ such that
at least one of the values $r_1,\ldots,r_{k},a_1,\ldots,a_{\ell}$ is $\alpha$,
at least one is $\beta$ and at least one is $\gamma$.
Consequently, the matrix $M$ contains the following submatrix possibly after dividing some columns by two (the columns correspond
to those of the variables $r_1,\ldots,r_{k},a_1,\ldots,a_{\ell}$ that are equal to $\alpha$, $\beta$ and $\gamma$, respectively;
the columns corresponding to the variables $a_1,\ldots,a_{\ell}$ are divided by two).
\[\begin{bmatrix}
  1 & 1 & 1 \\
  3\alpha^2 & 3\beta^2 & 3\gamma^2 \\
  4\alpha^3 & 4\beta^3 & 4\gamma^3
  \end{bmatrix}\]
The determinant of this matrix is equal to
\[12(\alpha^2\beta^3+\beta^2\gamma^3+\gamma^2\alpha^3-\alpha^2\gamma^3-\beta^2\alpha^3-\gamma^2\beta^3)\,,\]
which is equal to
\[12(\alpha-\beta)(\beta-\gamma)(\gamma-\alpha)(\alpha\beta+\alpha\gamma+\beta\gamma)\,.\]
Since this expression is non-zero for all distinct positive reals $\alpha$, $\beta$ and $\gamma$,
we conclude that the rank of $M$ is three, which contradicts our assumption that
the first case as described in the statement of the lemma does not apply.
\end{proof}

Before we can prove the main result of this section, we need two additional auxiliary lemmas.

\begin{lemma}
\label{lm-cubes}
Let $z\in (0,1]$, and
let $x_1,\ldots,x_n$ be non-negative reals such that $x_1+\cdots+x_n=1/2$ and $x_i\le z$ for every $i\in [n]$.
Then it holds that
\[\sum_{i=1}^n x_i^3\le\lfloor z^{-1}/2\rfloor\cdot z^3+(1/2-\lfloor z^{-1}/2\rfloor\cdot z)^3,\]
and the equality holds if and only if all but at most one of $x_1,\ldots,x_n$ are equal to $0$ or $z$.
\end{lemma}

\begin{proof}
Consider any $n$-tuple $x_1,\ldots,x_n$ that maximizes the sum $x_1^3+\cdots+x_n^3$
among all $n$-tuples of non-negative reals $x_1,\ldots,x_n$ such that $x_1+\cdots+x_n=1/2$ and $x_i\le z$, $i\in [n]$.
If $x_i\in\{0,z\}$ for all but at most one $i\in [n]$, then the sum of $x_1^3+\cdots+x_n^3$
is equal to $\sum_{i=1}^n x_i^3\le\lfloor z^{-1}/2\rfloor\cdot z^3+(1/2-\lfloor z^{-1}/2\rfloor\cdot z)^3$ and the lemma holds.
Otherwise, there exist $x_i$ and $x_j$ such that $0<x_i\le x_j<z$.
Choose $\varepsilon>0$ such that $\varepsilon<x_i$ and $\varepsilon<z-x_j$, and
replace $x_i$ with $x_i-\varepsilon$ and $x_j$ with $x_j+\varepsilon$.
This preserves the sum $x_1+\cdots+x_n$ and increases the sum $x_1^3+\cdots+x_n^3$,
which contradicts the choice of the $n$-tuple $x_1,\ldots,x_n$.
\end{proof}

Linial and Morgenstern~\cite{LinM16} proved that, among the random blow-ups of transitive tournaments with the fixed density of $C_3$,
the density of $C_4$ is minimized if all parts have the same size except possibly for a single smaller part.
This statement is equivalent to the following.

\begin{lemma}[{Linial and Morgenstern~\cite[Lemma 2.7]{LinM16}}]
\label{lm-linial}
Let $x_1,\ldots,x_n$ be any non-negative reals such that their sum is $1/2$.
It holds that
\[x_1^4+\cdots+x_n^4\ge g(x_1^3+\cdots+x_n^3)\,.\]
\end{lemma}

We are now ready to prove the main result of this section.

\begin{theorem}
\label{thm-reg23}
Let $A$ be a tournament matrix of order $n$.
If $\sigma_3(A)\ge 1/72$, then $\sigma_4(A)\ge g(\sigma_3(A))$.
\end{theorem}

\begin{proof}
Let $s_3=\sigma_3(A)$.
We start with lower bounding the spectral radius of $A$.
Let $\lambda_1,\ldots,\lambda_n$ be the eigenvalues of $A$.
By Lemma~\ref{lm-positive}, the real parts of all the eigenvalues are non-negative,
which implies that
\[s_3=\sum_{i=1}^n\left(\frac{\lambda_i}{n}\right)^3\le\sum_{i=1}^n\left(\frac{\Real\lambda_i}{n}\right)^3\,.\]
By Lemma~\ref{lm-cubes}, the last sum is at most
\[\lfloor\rho_A^{-1}/2\rfloor\cdot\rho_A^3+(1/2-\lfloor\rho_A^{-1}/2\rfloor\cdot\rho_A)^3\]
where $\rho_A$ is the spectral radius of $A$ divided by $n$.
Consequently, $\rho_A$ is at least $z$
where $z$ is the unique real between $0$ and $1/2$ satisfying that
\[s_3=\lfloor z^{-1}/2\rfloor\cdot z^3+(1/2-\lfloor z^{-1}/2\rfloor\cdot z)^3\,.\]
Note that $z\ge 1/6$ since $s_3\geq 1/72$.

Lemma~\ref{lm-problem} now yields that the theorem will be proven
if we show that the optimal solution of
the problem \spectrum{} is at least $g(s_3)$ for $s_3$, any $\rho\ge z$ and all non-negative integers $k$ and $\ell$.
By Lemma~\ref{lm-lagrange}, this would be implied by the following two claims,
which correspond to the two cases described in the statement of Lemma~\ref{lm-lagrange}.
\begin{description}
\item[Claim 1.] If $r_1,\ldots,r_k$ are any positive real numbers that have at most three distinct values and
               that satisfy $r_1+\cdots+r_k=1/2$ and $r_1^3+\cdots+r_k^3=s_3$, then $r_1^4+\cdots+r_k^4\ge g(s_3)$.
\item[Claim 2.] If $m$ and $m'$ are positive integers, $\rho\geq z$, $a$ is a non-negative real and $b$ is a real such that
               $m\rho+2m'a=1/2$ and $m\rho^3+2m'(a^3-3ab^2)=s_3$,
	       then $m\rho^4+2m'(a^4-6a^2b^2+b^4)\ge g(s_3)$.
\end{description}
Claim 1 follows from Lemma~\ref{lm-linial} (even without the restriction on the number of the distinct values that
$r_1,\ldots,r_k$ may have). So, we focus on proving Claim 2 in the remainder of the proof.
Note that it holds that $m\in\{1,2\}$ in this case since $\rho\ge z\ge 1/6$.

To prove Claim 2, we fix $m$ and $m'$ and consider the following optimization problem:
minimize the sum $m\rho^4+2m'(a^4-6a^2b^2+b^4)$ subject to $m\rho+2m'a=1/2$, $m\rho^3+2m'(a^3-3ab^2)=s_3$, $\rho\ge z$ and $a\ge 0$.
The method of Langrange multipliers implies that the following matrix is not full rank
\begin{equation}
  \begin{bmatrix}
  m & 2m' & 0 \\
  3m\rho^2 & 6m'a^2-6m'b^2 & -12m'ab \\
  4m\rho^3 & 8m'a^3-24m'ab^2 & 8m'b^3-24m'a^2b
  \end{bmatrix}
  \label{eq-lagrange2}
\end{equation}
for the values of $\rho$, $a$ and $b$ that minimize the sum
unless $\rho=z$ or $a=0$.
However, dividing the first column by $m$ and the remaining two by $m'$, permuting the columns and renaming the variables
yields the same matrix as in \eqref{eq-lagrange}, which we have analyzed in the proof of Lemma~\ref{lm-lagrange}.
In particular, the matrix in \eqref{eq-lagrange2} has full rank unless $a=0$ or (possibly) $b=0$.
We conclude that the expression $m\rho^4+2(a^4-6a^2b^2+b^4)$ is minimized
when at least one of the following holds: $\rho=z$, $a=0$ or $b=0$.
We next analyze these three cases.
\begin{description}
\item[The case $\rho=z$.]
In this case, Lemma~\ref{lm-cubes} implies that $m\rho^3+2m'a^3<s_3$ unless $a\geq \rho$,
i.e., there is no such feasible solution unless $a\geq \rho$.
If indeed $a \geq \rho$, then since $m\rho+2m'\rho=1/2$ and $\rho\geq1/6$, it must be that $m=m'=1$, $\rho=a=1/6$, $s_3=1/72$ and $b=0$, in which case $m\rho^4+2m'(a^4-6a^2b^2+b^4)=1/432=g(s_3)$.
\item[The case $a=0$.]
If $a=0$, then, as noted in the proof of Lemma \ref{lm-lagrange}, setting $b=0$ does not affect the constraints, but does decrease the objective function. Hence this case reduces to the final case $b=0$.

\item[The case $b=0$.]
If $b=0$, then using Lemma~\ref{lm-linial} we get that $m\rho^4+2m'a^4\ge g(m\rho^3+2m'a^3)=g(s_3)$.
\end{description}
Hence, we have shown that $m\rho^4+2m'(a^4-6a^2b^2+b^4)\ge g(s_3)$ for all $\rho$, $a$ and $b$ such that
$m\rho+2m'a=1/2$, $m\rho^3+2m'(a^3-3ab^2)=s_3$, $\rho\ge z$ and $a\ge 0$.
The proof of Claim 2 is now finished and so is the proof of the theorem.
\end{proof}

Proposition~\ref{prop-limit} yields the following corollary in the tournament limit setting.

\begin{corollary}
\label{cor-reg23}
Let $W$ be a tournament limit.
If $t(C_3,W)\ge 1/72$, then it holds that $t(C_4,W)\ge g(t(C_3,W))$.
\end{corollary}

\section{Concluding remarks}
\label{sec-concl}
Unfortunately, the proof of Theorem~\ref{thm-reg23} does not immediately work in higher regimes because, in the case of four or more parts, the solution to the optimization problem \spectrum{} beats the conjectured minimum. Of course, as solutions need not be realizable as the eigenvalues of a tournament matrix, this does not invalidate Conjecture~\ref{conj-main}. It is possible that the current method can be pushed further by introducing to the optimization problem additional constraints that reflect the properties that eigenvalues of tournament matrices must satisfy.

Meanwhile, in the regime of two parts, we have been able to fully determine the asymptotic structure of extremal examples.
These constructions can be extended to the remaining regimes as follows.
Fix $k\in\NN$, $z\in [1/(k+1),1/k]$ and $i,i'\in [k+1]$ such that $|i-i'|=1$.
We construct a tournament $T$ with $n$ vertices as follows.
The vertices of $T$ are split into $k+1$ parts $V_1,\ldots,V_{k+1}$ such that
$|V_i|=n-k\lfloor zn\rfloor$ and $|V_j|=\lfloor zn\rfloor$ if $j\neq i$.
If two vertices $v$ and $v'$ respectively belong to distinct parts $V_j$ and $V_{j'}$ with $j<j'$ and $\{j,j'\}\neq \{i,i'\}$,
then the tournament $T$ contains an arc from $v$ to $v'$.
If, instead, $v$ and $v'$ belong to the \textit{same} part $V_j$, where $j\not\in\{i,i'\}$,
then the edge between $v$ and $v'$ is oriented from $v$ to $v'$ with probability $1/2$,
i.e., the vertices of every such part induce a randomly oriented tournament.
Finally, each vertex $v\in V_i\cup V_{i'}$ is assigned a real number $p_v\in [0,1/2]$ and
the edge between $v$ and $v'\in V_i\cup V_{i'}$ is directed from $v$ to $v'$ with probability $1/2+p_v-p_{v'}$.
If the expected number of triangles in $T$
is equal to $\frac{1}{8}\left(k\lfloor zn\rfloor^3+\left(n-k\lfloor zn\rfloor\right)^3\right)$,
then the expected value of the density $t(C_4,T)$ is
\[g\left(\frac{1}{8}\left(k\lfloor zn\rfloor^3+\left(n-k\lfloor zn\rfloor\right)^3\right)\right)+o(1)\]
and both of these random variables are concentrated.
In particular, unless $z^{-1}$ is a positive integer,
there are infinitely many different types of extremal tournaments. As discussed in the introduction, we believe this is the reason why the problem is resistant to the standard flag algebra techniques.

It is also interesting to note that the problem of determining the set of feasible
densities of cycles of length three and four is equivalent to the analogous problems
for transitive tournaments of order three and four and for the cycle and transitive
tournament of order four~\cite[Proposition 1.1]{LinM16}. 
To see this, let $T_k$ be the transitive tournament of order $k$ and
let $t(T_k,T)$ be the probability that a random mapping from $V(T_k)$ to $V(T)$ is a homomorphism.
The following holds for every $n$-vertex tournament $T$:
\[8t(C_3,T)+24t(T_4,T)-6t(C_4,T)=1-O(n^{-1}).\]
Thus, the problem of minimizing the density of $C_4$ when the density of $C_3$ is fixed is
equivalent to minimizing the density of $T_4$ when the density of $T_3$ is fixed, and
also equivalent to minimizing the density of $C_4$ when $T_4$ is fixed, in the sense that
a complete solution to any of these three problems yield complete solutions to the remaining two.

\bibliographystyle{bibstyle}
\bibliography{tourn-c4}

\begin{thebibliography}{10}
\providecommand{\url}[1]{\texttt{#1}}
\providecommand{\urlprefix}{URL }
\providecommand{\eprint}[2][]{\url{#2}}

\bibitem{Bol76}
B.~Bollob\'{a}s: \emph{Relations between sets of complete subgraphs}, in:
  Proceedings of the {F}ifth {B}ritish {C}ombinatorial {C}onference ({U}niv.
  {A}berdeen, {A}berdeen, 1975) (1976), 79--84.

\bibitem{BraG68}
A.~Brauer and I.~C. Gentry: \emph{On the characteristic roots of tournament
  matrices}, Bull. Amer. Math. Soc. \textbf{74} (1968), 1133--1135.

\bibitem{ChuG91}
F.~R.~K. Chung and R.~L. Graham: \emph{Quasi-random tournaments}, J. Graph
  Theory \textbf{15} (1991), 173--198.

\bibitem{Erd55}
P.~Erd\H{o}s: \emph{Some theorems on graphs}, Riveon Lematematika \textbf{9}
  (1955), 13--17.

\bibitem{Fis89}
D.~C. Fisher: \emph{Lower bounds on the number of triangles in a graph}, J.
  Graph Theory \textbf{13} (1989), 505--512.

\bibitem{FriK99}
A.~Frieze and R.~Kannan: \emph{Quick approximation to matrices and
  applications}, Combinatorica \textbf{19} (1999), 175--200.

\bibitem{Gan98}
F.~R. Gantmacher: The theory of matrices. {V}ol. 1, 1998, translated from the
  Russian by K. A. Hirsch, Reprint of the 1959 translation.

\bibitem{Goo59}
A.~W. Goodman: \emph{On sets of acquaintances and strangers at any party},
  Amer. Math. Monthly \textbf{66} (1959), 778--783.

\bibitem{LinM16}
N.~Linial and A.~Morgenstern: \emph{On the number of 4-cycles in a tournament},
  J. Graph Theory \textbf{83} (2016), 266--276.

\bibitem{LiuPS17}
H.~Liu, O.~Pikhurko and K.~Staden: \emph{The exact minimum number of triangles
  in graphs of given order and size}, preprint arXiv:1712.00633  (2017).

\bibitem{Lov12}
L.~Lov\'asz: Large Networks and Graph Limits, \emph{Colloquium Publications},
  volume~60, 2012.

\bibitem{LovS83}
L.~Lov\'{a}sz and M.~Simonovits: \emph{On the number of complete subgraphs of a
  graph. {II}}, in: Studies in pure mathematics (1983), 459--495.

\bibitem{LovS06}
L.~Lov\'asz and B.~Szegedy: \emph{Limits of dense graph sequences}, J. Combin.
  Theory Ser. B \textbf{96} (2006), 933--957.

\bibitem{Man07}
W.~Mantel: \emph{Problem 28}, Wiskundige Opgaven \textbf{10} (1907), 60--61.

\bibitem{Nik11}
V.~S. Nikiforov: \emph{The number of cliques in graphs of given order and
  size}, Trans. Amer. Math. Soc. \textbf{363} (2011), 1599--1618.

\bibitem{PikR17}
O.~Pikhurko and A.~Razborov: \emph{Asymptotic structure of graphs with the
  minimum number of triangles}, Combin. Probab. Comput. \textbf{26} (2017),
  138--160.

\bibitem{Raz08}
A.~A. Razborov: \emph{On the minimal density of triangles in graphs}, Combin.
  Probab. Comput. \textbf{17} (2008), 603--618.

\bibitem{Rei16}
C.~Reiher: \emph{The clique density theorem}, Ann.~of~Math.~(2)  (2016),
  683--707.

\end{thebibliography}

\end{document}